\documentclass[a4paper,12pt,reqno]{amsart}
\usepackage{amssymb}
\usepackage{amsmath}
\usepackage{mathtools}
\usepackage{enumitem}
\usepackage{mathrsfs}
\usepackage[all]{xy}
\usepackage{mathtools}
\usepackage{hyperref}
\usepackage{url}
\usepackage{bbm}

\usepackage[british]{babel}
\usepackage[margin=1.2in]{geometry}
\numberwithin{equation}{section} \numberwithin{figure}{section}

\DeclareMathOperator{\Pic}{Pic} 
\DeclareMathOperator{\Gal}{Gal} 
\DeclareMathOperator{\Aut}{Aut} 
\DeclareMathOperator{\Spec}{Spec}

\DeclareMathOperator{\Br}{Br} 

\DeclareMathOperator{\inv}{inv}

\DeclareMathOperator{\Bl}{Bl}

\DeclareMathOperator{\chr}{char}

\DeclareMathOperator{\Frob}{Fr} 
\DeclareMathOperator{\HH}{H}

\newcommand{\OO}{\mathcal{O}}

\newcommand{\PGL}{\textrm{PGL}}

\newcommand{\Esix}{{\mathbf E}_6}
\newcommand{\Eseven}{{\mathbf E}_7}

\newcommand{\E}{{\mathbf E}}

\newcommand\FF{\mathbb{F}}
\newcommand\PP{\mathbb{P}}
\newcommand\ZZ{\mathbb{Z}}
\newcommand\NN{\mathbb{N}}
\newcommand\QQ{\mathbb{Q}}

\newcommand\CC{\mathbb{C}}

\newcommand{\fp}{\mathfrak{p}}

\newtheorem{lemma}{Lemma}

\newtheorem{theorem}[lemma]{Theorem}
\newtheorem{proposition}[lemma]{Proposition}
\newtheorem{corollary}[lemma]{Corollary}
\numberwithin{table}{section}

\theoremstyle{definition}

\newtheorem{definition}[lemma]{Definition}
\newtheorem{remark}[lemma]{Remark}

\newtheorem*{ack}{Acknowledgements}

\numberwithin{lemma}{section}

\begin{document}

\title[Del Pezzo surfaces over finite fields]
{Del Pezzo surfaces over finite fields and their Frobenius traces}

\author{\sc Barinder Banwait}
\address{Barinder Banwait \\
CMR Surgical,  
Crome Lea Business Park\\
Madingley Road, Cambridge,
CB23 7PH, UK.}
\email{b.s.banwait86@gmail.com}

\author{\sc Francesc Fit\'{e}}
\address{Francesc   Fit\'{e} \\
Departament de Matem\`atiques\\ Universitat Polit\`ecnica de    Catalunya\\ BGSmath\\
Edifici Omega, C/Jordi Girona 1--3\\
08034 Barcelona\\ 
Catalonia.}
\email{francesc.fite@gmail.com}
\urladdr{https://mat-web.upc.edu/people/francesc.fite/}

\author{\sc Daniel Loughran}
\address{Daniel Loughran \\
School of Mathematics \\
University of Manchester \\
Oxford Road \\
Manchester \\
M13 9PL \\
UK.}
\email{daniel.loughran@manchester.ac.uk}
\urladdr{https://sites.google.com/site/danielloughran/}

\subjclass[2010]
{14G15  (primary), 
14G05, 
14J20.   
(secondary)}

\begin{abstract}
Let $S$ be a smooth cubic surface over a finite field $\FF_q$. It is known that $\#S(\FF_q) = 1 + aq + q^2$ for some $a \in \{-2,-1,0,1,2,3,4,5,7\}$. Serre has asked which values of $a$ can arise for a given $q$. Building on special cases  treated by Swinnerton-Dyer, we give a complete answer to this question. We also answer the analogous question for other del Pezzo surfaces, and consider the inverse Galois problem for del Pezzo surfaces over finite fields. Finally we give a corrected version of Manin's and Swinnerton-Dyer's tables on cubic surfaces over finite fields.
\end{abstract}

\maketitle

\tableofcontents

\section{Introduction}
\subsection{A question of Serre}
Let $S$ be a smooth cubic surface over a finite field $\FF_q$. It is well known (see for example \cite[Thm.~27.1]{Man86}) that
$$\#S(\FF_q) = 1 + a(S)q + q^2$$
where $a(S)$ is the trace of the Frobenius element $\Frob_q \in \Gal(\bar{\FF}_q/\FF_q)$ acting on the Picard group $\Pic \bar{S} \cong \ZZ^7$ of $\bar{S}:= S_{\bar{\FF}_q}$.

For a choice of isomorphism $\Aut(\Pic \bar{S}) \cong W(\E_6)$ with the Weyl group of the $\E_6$-root system, the action of $\Frob_q$ yields a conjugacy class of  $W(\E_6)$  (we only consider those automorphisms of $\Pic \bar{S}$ which preserve the canonical class and the intersection pairing; see \cite[Thm.~23.9]{Man86}). An inspection of the character table of $W(\E_6)$ reveals that we have $a(S) \in \mathscr{A}_3:=\{-2,-1,0,1,2,3,4,5,7\}$ (see \cite[\S2.3.3]{Ser12}). 
Serre asked in \emph{loc.~cit.}~which values of the trace can actually arise for a given $q$. Swinnerton-Dyer \cite{SD10} has shown that the trace values $-2,5$ arise for all $q$, whereas $7$ occurs if and only if $q \neq 2,3,5$. We extend this to give a complete answer to Serre's question.

\begin{theorem} \label{thm:dp3}
Let $q$ be a prime power.
\begin{enumerate}
\item
For $-2 \leq a \leq 5$ and for all $q$, there exists a smooth cubic surface $S$ over $\FF_q$ with $a(S) = a$.
\item
There exists a smooth cubic surface $S$ over $\FF_q$ with $a(S) = 7$ if and only if $q \neq 2,3,5$.
\end{enumerate}
\end{theorem}
Note that a smooth cubic surface $S$ with $a(S)=7$ is \emph{split}, i.e.~all its lines are defined over $\FF_q$. It has been known for a long time, prior to the work of Swinnerton-Dyer \cite{SD10}, that such a surface exists over $\FF_q$ if and only if $q \neq 2,3,5$; this result appears to be first due to Hirschfeld \cite[Thm.~20.1.7]{Hir85}.

We briefly explain the proof of Theorem \ref{thm:dp3}.
Any smooth cubic surface over an \emph{algebraically closed field} is the blow-up of $\PP^2$ in $6$ rational points in general position. Whilst this does not hold over other fields in general, the trace values $1,2,3,4,5$ can be obtained from cubic surfaces which are blow-ups of $\PP^2$ in collections of \emph{closed} points in general position. 
We therefore  show that such collections exist over every finite field $\FF_q$,  via combinatorial arguments. Trace $0$ can be obtained by blowing-up certain collections of closed points of total degree $7$, and contracting a line. The existence of the remaining traces $-1,-2$ can be deduced from work of Rybakov \cite{Ryb05} and Swinnerton-Dyer \cite{SD10}, respectively.

\subsection{Corrections to Manin's and Swinnerton-Dyer's tables}
Let $S$ be a smooth cubic surface over a finite field $\FF_q$. Building on work of Frame \cite{Fra51} and Swinnerton-Dyer \cite{SD67}, Manin constucted a table (Table $1$ of \cite[p.~176]{Man86}) of the conjugacy classes of $W(\Esix)$ and their properties, such as the trace $a(S)$.

Urabe \cite{Ura96a,Ura96b} was the first to notice that Manin's table contains mistakes regarding the calculation of $\HH^1(\FF_q, \Pic \bar{S})$ (the issue being that $\HH^1(\FF_q, \Pic \bar{S})$ must have \emph{square} order). In our investigation we found some new mistakes. These concern the Galois orbit on the lines, where the mistake can be traced back to Swinnerton-Dyer \cite{SD67}, and the \emph{index} \cite[\S28.2]{Man86} of the surface. The index is the size of the largest Galois invariant collection of pairwise skew lines over $\bar{\FF}_q$.

Manin's table has a surface of index $2$, which led him to state \cite[Thm.~28.5(i)]{Man86} that the index can only take one of the values $0,1,2,3,6$. Our investigations reveal, however, that index $2$ does not occur, and that index $5$ can occur, hence the correct statement is the following.

\begin{theorem}
	Let $S$ be a smooth cubic surface over a finite field. Then the index of $S$ can
	only take one of the values $0,1,3,5,6$.
\end{theorem}

A corrected table can be found in Section \ref{sec:table}. We constructed this table using \texttt{Magma}; we also give geometric proofs of our corrected values for completeness.

Manin's book initiated a wave of interest in the arithmetic of del Pezzo surfaces, which
continues to this day. The authors hope that it will be a useful addition to the literature to include a fully corrected version of Manin's table over $40$ years after it was originally published. 

\subsection{Del Pezzo Surfaces}
We also consider the analogue of Serre's question for other \emph{del Pezzo surfaces} of degree $d$ (cubic surfaces being the del Pezzo surfaces of degree $3$; see \S \ref{sec:del_Pezzo} for definitions). Knowledge about the possible number of rational points on del Pezzo surfaces over finite fields is often required in proofs, especially over small finite fields (see e.g.~the proof of \cite[Thm.~1]{STVA14}). The arithmetic of del Pezzo surfaces becomes more difficult as the degree decreases, and we focus here on the cases $d \leq 4$. One has \cite[Thm.~23.9]{Man86} an isomorphism $\Aut(\Pic \bar{S}) \cong W(\E_{9-d})$  (we follow Dolgachev's convention \cite[\S8.2.3]{Dol12}, and define the $\E_r$-root system for any $3 \leq r \leq 8$). As in the case of cubic surfaces, one can use the character table of $W(\E_{9-d})$ to see  that the trace of Frobenius $a(S)$ belongs to $\mathscr{A}_d$, where
\begin{align} \label{def:A_d}
\begin{split}
	\mathscr{A}_4 & = \{-2,-1,0,1,2,3,4,6\}, \\
	\mathscr{A}_2 & = \{-6,-4,-3,-2,-1,0,1,2,3,4,5,6,8\}, \\
	\mathscr{A}_1 & = \{-7,-5,-4,-3,-2,-1,0,1,2,3,4,5,6,7,9\}.
\end{split}
\end{align}
The set $\mathscr{A}_4$ can be deduced from Table \ref{tab:Manin}:
the blow-up of a quartic del Pezzo surface $S$ in a rational point not on a line is a cubic surface $S'$ with a line,
and one has $a(S') = a(S) + 1$ by Lemma \ref{lem:blow_up} (the rows of Table \ref{tab:Manin} corresponding to cubic surfaces with a line are exactly those whose orbit type contains $1^b$ for some $b \in \NN$). The values for $\mathscr{A}_2$ and $\mathscr{A}_1$ can be found in 
Urabe's tables \cite{Ura96a} (the trace $a$ is the exponent $1^a$ of $1$ of the Frame symbol in \emph{loc.~cit.}).
We give a complete classification of which traces can arise.

\begin{theorem} \label{thm:dp4}
Let $q$ be a prime power.
\begin{enumerate}
\item
For $-2 \leq a \leq 4$ and for all $q$, there exists a quartic del Pezzo surface $S$ over $\FF_q$ with $a(S) = a$.
\item
There exists a quartic del Pezzo surface $S$ over $\FF_q$ with $a(S) = 6$ if and only if $q \neq 2,3$.
\end{enumerate}
\end{theorem}

Theorem \ref{thm:dp4} follows fairly readily from the method used to handle cubic surfaces. The following results require more work.

\begin{theorem} \label{thm:dp2}
	For each $a \in \mathscr{A}_2$ let $B_a$ denote the set of prime powers $q$ for which there does not exist a del Pezzo surface $S$ of degree $2$ over $\FF_q$ with $a(S) = a.$ Then 
$$
\begin{array}{lll}
(1) \ B_a = B_{2-a}; & 
(2) \ B_a = \emptyset, \, a = 1,2,3; &
(3) \ B_a=\left\{2\right\}, \, a=4,5;  \\
(4) \ B_6 = \left\{2,3,4\right\}; &
(5) \ B_8 = \left\{2,3,4,5,7,8\right\}.& 
\end{array}$$
	
\end{theorem}

\begin{theorem} \label{thm:dp1}
	For each $a \in \mathscr{A}_1$ 	
	let $B_a$ denote the set of prime powers $q$ for which there does not exist a del Pezzo surface $S$ of degree $1$ over $\FF_q$ with $a(S) = a.$ Then 
$$
\begin{array}{lll}
(1) \ B_a = B_{2-a}; & 
(2) \ B_a = \emptyset, \, a = 1,2,3,4; &
(3) \ B_5 = \left\{2\right\};  \\
(4) \ B_6 = \left\{2,3,4,5\right\}; &
(5) \ B_7 = \left\{2,3,4,5,7,8,9\right\};&  \\
(6) \ B_9 = \{2,3,4,5,7,8,9,11,13,17\}.  &  &
\end{array}
$$

\end{theorem}

The cases $a \geq 1$ are handled similarly to the proof of Theorem \ref{thm:dp3}, by considering various configurations of closed points in general position in $\PP^2$ (at least this works for large $q$).  For small $q$ we also need to consider some surfaces which are not of maximal index, and show either their existence or non-existence. We do this using a range of combinatorial and geometric techniques, such as the theory of conic bundles and  the classification of possible Galois actions by Urabe \cite{Ura96a}. \texttt{Magma} is employed to help with a few remaining difficult cases.
When the required points in general position exist, one finds them on a computer very quickly, as a ``randomly'' chosen collection of points will lie in general position. For very small $q$ (e.g.~$q<5$) we can often rule out trace values by pure thought, but for slightly higher  values (e.g.~$5 \leq q \leq 9$), the value of $q$ is still small enough to prove non-existence by enumerating all possibilities. (The most computationally intensive case was proving the non-existence of $6$ rational points and a closed point of degree $2$ in general position when $q=9$, which took about $90$ minutes on a desktop computer). Note that, at no point in our arguments do we need to enumerate \emph{all} del Pezzo surfaces of degree $2$ or $1$ in order to verify the non-existence of certain traces; our approach of enumerating configurations of closed points in $\PP^2$ in general position and appealing to the classification of conjugacy classes due to Urabe \cite{Ura96a} is substantially faster.

For $a < 1$, we use an amusing trick: any del Pezzo surface $S$ of degree $2$ or $1$ admits a special automorphism of order $2$, hence a non-trivial quadratic twist over $\FF_q$, which we denote by $S_\sigma$. A counting argument shows that $a(S) + a(S_\sigma) = 2$, in particular, on performing a quadratic twist we handle the remaining trace values (this explains the above symmetry $B_a = B_{2-a}$).

Note that a del Pezzo surface $S$ over $\FF_q$ of degree $d \leq 7$ has trace $10-d$ if and only if $S$ is the blow-up of $\PP^2$ in $9-d$ rational points in general position. In particular, for $r \leq 8$, our results give a complete classification of those $q$ for which there exist $r$ rational points in $\PP^2$ over $\FF_q$ in general position. The new cases are as follows.

\begin{corollary} \
\begin{enumerate}
\item
$\PP^2_{\FF_q}$ has  $7$ rational points in general position if and only if $q \geq 9$. 
\item
$\PP^2_{\FF_q}$ has  $8$ rational points in general position  if and only if $q=16$ or $q \geq 19$.
\end{enumerate}
\end{corollary}

\subsection{An inverse Galois problem}
Generalising Serre's question, one may ask which
conjugacy classes of $W(\E_{9-d})$ arise from the Galois action on some del Pezzo surface of degree $d$ over $\FF_q$. We are able to answer
this question for sufficiently large $q$. For uniformity of exposition we focus on the more interesting case $d \leq 6$.

For a field $k$, we denote by $\mathcal{S}_d(k)$
the set of isomorphism classes of del Pezzo surfaces of degree $d$ over $k$.
For a del Pezzo surface $S$ of degree $d$ over a finite field,
we denote by $C(S)$ the conjugacy class of $W(\E_{9-d})$ obtained from the action of $\Frob_q$ on $\Pic \bar{S}$.
We want to study the distribution of $C(S)$ as $q$ grows. As is common when counting objects up to isomorphism, we weight
each surface by the size of its automorphism group. Our result is as follows.

\begin{theorem} \label{thm:inverse}
	For $d \leq 6$ we have 
	$$\lim_{q \to \infty}\frac{\sum_{S \in \mathcal{S}_d(\FF_q), C(S) = C}
	\frac{1}{|\Aut S|}}{\sum_{S \in \mathcal{S}_d(\FF_q)}\frac{1}{|\Aut S|}} = \frac{\#C}{\#W(\E_{9-d})}.$$
\end{theorem}

Values for $\#C/\#W(\E_{9-d})$ when $d=3,2,1$ can be found in Table \ref{tab:Manin}, and  \cite[Tab.~1, Tab.~2]{Ura96a}, respectively. We prove Theorem \ref{thm:inverse} using known results on the monodromy groups of generic del Pezzo surfaces, together with a  version of the Chebotarev density theorem due to Ekedahl \cite{Eke98}, which is proved using Deligne's ``Weil II paper'' \cite{Del80}. 
From this we immediately obtain the following.

\begin{corollary} \label{cor:inverse}
	For $q \gg 1$, the inverse Galois problem for del Pezzo surfaces of degree $d$ over $\FF_q$ is solvable.
\end{corollary}

Of course Corollary \ref{cor:inverse} is only new for small $d$. For $d \geq 5$ it is known
that every conjugacy class is realisable over every finite field (see e.g.~\cite[Thm.~3.1.3]{Sko01} for $d=5$ and \cite[Thm.~3.5]{Blu10}, \cite[Thm.~4.2]{CTKM} for $d=6$, respectively). For $d\leq 4$, however, Corollary \ref{cor:inverse} appears to be new. Note that the resolution of the inverse Galois problem for cubic surfaces over $\QQ$ is a recent result of Elsenhans-Jahnel \cite{EJ15}.

The proof of Corollary \ref{cor:inverse} could in theory be made effective. In the case $d=3$, for example, one would need upper bounds for the dimensions of cohomology groups with compact support for certain $\ell$-adic sheaves on  the open subset $U \subset \PP^{19}$ that parametrises smooth cubic surfaces in $\PP^3$. It seems doubtful however that such bounds would be good enough for the remaining cases to be amenable to machine computation for small $d$. The problem of constructing minimal del Pezzo surfaces of degree $3$ and $2$ over finite fields is considered in recent work of Rybakov-Trepalin \cite{RT16} and Trepalin \cite{Tre16}, respectively. Note that of course the conclusion of Corollary~\ref{cor:inverse}
does not hold for all $q$, as Theorem \ref{thm:dp3} illustrates.

For $d\leq 3$, we are able to count without weighting by the automorphism group.
\begin{theorem} \label{thm:inverse_2}
	For $d\leq 3$ we have 
	$$ \lim_{q \to \infty}\frac{\#\{ S \in \mathcal{S}_d(\FF_q) : C(S) = C\}}{\#\mathcal{S}_d(\FF_q)} = \frac{\#C}{\#W(\E_{9-d})}.$$
\end{theorem}

\subsubsection{Vertical Sato--Tate}
As an application, we may address  the distribution of  trace values for del Pezzo surfaces over $\FF_q$, as $q \to \infty$. For $a \in \ZZ$ let

\begin{equation}  \label{eqn:tau(a)}
	\tau_{d}(a) = \lim_{q\to\infty} \frac{\#\left\{S \in \mathcal{S}_d(\FF_q) : a(S) = a\right\}}
	{\# \mathcal{S}_d(\FF_q)}.
\end{equation}
Theorem \ref{thm:inverse_2} and an enumeration of the conjugacy classes of $W(\E_{9-d})$ yields the following.

\begin{corollary}
For $d \leq 3$, the limit \eqref{eqn:tau(a)} exists and takes the following values:
$$\begin{array}{c|c|c|c}
		a & \tau_3(a) & \tau_2(a) & \tau_1(a) \\
		\hline
		-7 & & & 1/696729600 \\
		-6 & & 1/2903040 & \\
		-5 & & & 1/5806080 \\
		-4 & & 1/46080 & 1/311040\\
		-3 & & 1/4320 & 653/4976640\\
		-2 &1/648 & 13/3072 & 2267/518400 \\
		-1 &  77/1152 & 169/3240 & 225157/4147200\\
		0 & 9/40 & 34423/138240 & 262679/1088640 \\
		1 & 347/864 &653/1680 & 442169/1105920\\
		2 & 91/360 & 34423/138240 & 262679/1088640\\
		3 & 3/64 & 169/3240 & 225157/4147200\\
		4 & 1/216 & 13/3072 & 2267/518400 \\
		5 & 1/1440 & 1/4320 & 653/4976640\\
		6 & & 1/46080 & 1/311040 \\
		7 & 1/51840 & & 1/5806080\\
		8 & & 1/2903040 & \\
		9 & &  & 1/696729600
		\end{array}$$

\end{corollary}

\begin{ack}
	All computations were performed in \texttt{Magma}
	\cite{Magma}; the code can be found on the authors' web pages. 
	We thank Martin Bright, Andreas Enge, Tom Fisher, David Holmes, Ronald van Luijk, Cec\'{i}lia Salgado and Tony V\'{a}rilly-Alvarado
	for useful discussions. 
	In particular, we are grateful to Cec\'{i}lia for suggesting that we look at quadratic twists.	
	Thanks to Simeon Ball for alerting us to the work of Hirschfeld \cite{Hir79, Hir85}, and to J\"{o}rg Jahnel
	for assistance with some	\texttt{Magma} calculations. 
	Banwait and Fit\'{e} acknowledge the financial support of SFB/Transregio 45. This project has received funding from the European Research Council (ERC) under the European Union's Horizon 2020 research and innovation programme (grant agreement No 682152).
\end{ack}

\section{Generalities} \label{sec:generalities}

\subsection{The $a$-invariant}

\begin{definition}
	Let $S$ be a smooth geometrically rational projective surface over a finite field $\FF_q$. 
	We let $a(S)$  be the trace of the Frobenius element of $\Gal(\bar{\FF}_q/\FF_q)$
	acting on $\Pic \bar{S}$.
\end{definition}

\noindent
By \cite[Thm.~27.1]{Man86} we have
\begin{equation}		 \label{eqn:Weil}
	\#S(\FF_q) = 1 + a(S)q + q^2.
\end{equation}
In particular $a(S)$ plays a similar r\^{o}le
to the $a$-invariant of an elliptic curve.

\begin{lemma} \label{lem:blow_up}
	Let $S$ be a smooth geometrically rational projective surface over a finite field $\FF_q$.
	Let $x \in S$ be a closed point and $\Bl_x S$ the blow-up of $S$ at $x$. Then	
	$$
	a(\Bl_x S) =
	\begin{cases}
		a(S) + 1, & x \in S(\FF_q), \\
		a(S), & x \notin S(\FF_q).	
	\end{cases}
	$$
\end{lemma}
\begin{proof}
	If $x \in S(\FF_q)$, then blowing-up replaces a rational point by a copy of $\PP^1$, which has $1 + q$ rational points.
	If $x \notin S(\FF_q)$, then blowing-up neither removes nor adds any rational points.
\end{proof}

\subsection{Del Pezzo surfaces} \label{sec:del_Pezzo}
We recall some facts about del Pezzo surfaces, which can be found in \cite{Man86} and \cite[\S8]{Dol12}.
A del Pezzo surface $S$ over a field $k$ is a smooth projective surface over $k$ with ample anticanonical bundle.
We define the degree $d$ of $S$ to be $(-K_S)^2$; we have $1 \leq d \leq 9$. A \emph{line} on $S$ is a smooth geometrically
rational curve $L \subset S$ with $L^2 = -1$. For $3 \leq d \leq 7$, such a curve is a line in the usual sense with respect
to the anticanonical embedding $S \subset \PP^d$. We say that $S$ is \emph{split} if the natural map $\Pic S \to \Pic \bar{S}$ is an isomorphism; for $d\leq 7$ this is equivalent to all the lines of $S$ over $\bar{k}$ being defined over $k$.

\begin{definition} \label{def:gen_pos}
	Let $r \leq 8$ and let $k$ be a field. We say that a collection of distinct points $P_1, \ldots, P_r \in \PP^2(k)$
	lie in general position if the following hold.
	\begin{enumerate}
		\item No $3$ are collinear.
		\item No $6$ lie on a conic.
		\item No $8$ lie on a cubic with a singularity at one of the points.
	\end{enumerate}
	We say that a collection of distinct \emph{separable closed points} of $\PP^2$ of 
	total degree less than $8$ lie in general position
	if the corresponding points over $\bar{k}$ lie in general position
	(we say that a closed point is separable if its residue field is separable).
\end{definition}

We will often abuse notation and identify a separable closed point $P$ with the Galois
invariant collection $P_{k^{\mathrm{sep}}}$ of rational points over $k^{\mathrm{sep}}$.

As proved by Manin \cite[Thm.~24.5]{Man86} for $r \leq 6$ and D\'{e}mazure \cite[Thm.~1]{Dem80} for $r=7,8$,
a collection of $r \leq 8$ rational points lie in general position if and only if
their blow-up is a del Pezzo surface $S$. For $r \leq 7$, the lines on $S$ consist of the exceptional
curves of the blow-ups of the points, together with the 
strict transforms of the following curves \cite[Thm.~26.2]{Man86}:
\begin{enumerate}
	\item Lines through two of the points.
	\item Conics through five of the points.
	\item Cubic curves through seven of the points, with a double point at exactly one of them.
\end{enumerate}
There is a simple criterion to check whether $5$ points lie in general position.

\begin{lemma} \label{lem:general_pos}
	Let $P_1, \ldots, P_5 \in \PP^2(k)$ be five distinct points. Then $P_1, \ldots, P_5$ lie in general position
	if and only if they lie on a smooth conic.
\end{lemma}
\begin{proof}
	Clearly $P_1, \ldots, P_5$ lie on some conic $C$. If $C$ is singular, then it is either a union of two lines
	or a double line; in particular $3$ of the points are collinear. If $C$ is smooth, then any 
	line intersects $C$ in at most $2$ points, hence no $3$ are collinear.
\end{proof}

Our next lemma is a simple linear algebra criterion to check whether a collection of points lies in general position,
which will be used for computations. 

\begin{lemma} \label{lem:general_pos_det}
	Let $k$ be a field and let $P_i = [x_i:y_i:z_i] \in \PP^2(k)$ be a collection of distinct
	rational points, for $1 \leq i \leq r$.
	\begin{enumerate}
	
		\item If $r=3$, then $P_1,P_2,P_3$ are collinear if and only if
		\begin{equation}\label{mat:collinear}
		\left|\begin{matrix}
		x_1 & y_1 & z_1\\
		x_2 & y_2 & z_2\\
		x_3 & y_3 & z_3
		\end{matrix}
		\right|=0.
		\end{equation}
		
		\item If $r=6$, then $P_1,\ldots,P_6$ lie on a conic if and only if the matrix $M \in M_{6,6}(k)$
		whose $i$th row is
		\begin{equation}\label{mat:conic}
		(\begin{matrix}
		x_i^2 &y_i^2 & z_i^2 & x_iy_i & x_iz_i & y_iz_i
		\end{matrix})
		\end{equation}
		has determinant $0$.
		
		\item If $r=8$, for each $1 \leq i \leq 8$ consider the matrix $M_i \in M_{11,10}(k)$ whose $j$th row 
		for $1 \leq j \leq 8$ is
		$$\begin{matrix}
		(x_j^3 & y_j^3 & z_j^3 & x_j^2y_j & x_j^2z_j & x_jy_j^2 & y_j^2z_j & x_jz_j^2 & y_jz_j^2 & x_jy_jz_j),
		\end{matrix}$$
		and whose last three rows are
		\begin{equation*}\label{mat:cubic}
		\left(\begin{matrix}
		3x_i^2 & 0 & 0 & 2x_iy_i & 2x_iz_i & y_i^2 & 0 & z_i^2 & 0 & y_iz_i\\
		0 & 3y_i^2 & 0 & x_i^2 & 0 & 2x_iy_i & 2y_iz_i & 0 & z_i^2 & x_iz_i\\
		0 & 0 & 3z_i^2 & 0 & x_i^2 & 0 & y_i^2 & 2x_iz_i & 2y_iz_i & x_iy_i
		\end{matrix}\right).
		\end{equation*}
		Then $P_1,\ldots,P_8$ lie on a cubic with a singularity at $P_i$ if and only if
		$M_i$ has a non-trivial kernel.
	\end{enumerate}
\end{lemma}
\begin{proof}
	Parts $(1)$ and $(2)$ are elementary. For $(3)$, recall that a plane cubic curve $C$
	has $10$ coefficients. The condition $P_i \in C$ imposes linear relations on the coefficients.
	A vector that lies in the kernel
	of the first $8$ rows of $M_i$ corresponds to a cubic curve
	which contains the given $8$ points. The last three rows determine whether the partial
	derivatives of this cubic vanish at the point $P_i$, whence the result.
\end{proof}

\subsection{Conic bundles}
In this paper, a conic bundle over a field $k$ is a smooth
projective surface $S$ over $k$ together with a morphism $\pi:S \to \PP^1$ all of whose fibres are
isomorphic to conics. A fibre of $\pi$ is either a smooth conic or isomorphic over the algebraic closure to two lines meeting in a single point (note that a ``double line'' cannot occur as $S$ is non-singular; see \cite[Lem.~6]{Isk79}). We say that $\pi:S \to \PP^1$ is \emph{relatively minimal} if the fibre over every point is irreducible.

\begin{lemma} \label{lem:conic_bundle}
	Let $\pi:S \to \PP^1$ be a relatively minimal conic bundle over $\FF_q$. Then
	\begin{enumerate}
		\item The set
		$$\{ x \in \PP^1 : \pi^{-1}(x) \text{ is singular} \}$$
		has even cardinality. 
		\item We have
		$$a(S) = 2 - \#\{x \in \PP^1(\FF_q) : \pi^{-1}(x) \text{ is singular} \}.$$	
	\end{enumerate}
\end{lemma}
\begin{proof}
	Part $(1)$ follows from the fundamental exact sequence from class field theory for $\FF_q(t)$; see \cite[Cor.~2.10]{Ryb05}.
	Part $(2)$ follows from \eqref{eqn:Weil} and an elementary count: as $\pi$ relatively minimal,
	the fibre over a rational point is either a smooth plane conic, hence contains $1+q$ rational points, or is singular
	and contains exactly $1$ rational point (being $2$ lines over $\FF_{q^2}$ meeting in a single point).
\end{proof}

\section{Cubic surfaces and quartic del Pezzo surfaces} \label{sec:proofscubic}
In this section we prove Theorem \ref{thm:dp3}, following the strategy outlined in the introduction.
For trace at least $1$, it suffices to consider the existence of the following collections of closed points in general position.

$$\begin{array}{c|c|c}
	a & \text{No.} & \mbox{Points in general position to blow-up}\\
	\hline
	7 & 25 & \mbox{$6$ rational points}\\
	\hline
	5 & 24 & \mbox{$4$ rational points, one closed point of degree 2}\\
	\hline
	4 & 22 & \mbox{$3$ rational points, one closed point of degree $3$}\\
	\hline
	3 & 19 & \mbox{$2$ rational points, one closed point of degree $4$}\\
	\hline
	2 & 16 & \mbox{$1$ rational point, one closed point of degree $5$}\\
	\hline
	1 & 15 & \mbox{One closed point of degree $6$}
	\end{array}$$

The reader may verify with Lemma \ref{lem:blow_up} that the given blow-ups yield the claimed trace $a$. We have also included the number of the corresponding conjugacy class, as can be found in Table \ref{tab:Manin}. 
We briefly explain at the end the proof of Theorem \ref{thm:dp4}.

\subsection{Proof of Theorem \ref{thm:dp3}}

Whilst the cases $a=7,5$ were already dealt with by Swinnerton-Dyer \cite{SD10}, we give alternative proofs in the spirit of our method. The lemmas proved here will also be required in the sequel.
\subsection*{$a=7$}
First note that smooth cubic surfaces of trace $7$ are exactly those which are blow-ups of $\PP^2$ in $6$ rational points
in general position.

\begin{lemma}\label{lem:five_points}
There exist five points $P_1,\dots,P_5\in \PP^2(\FF_q)$ in general position if and only if $q\geq 4$.
\end{lemma}
\begin{proof}
This follows from Lemma \ref{lem:general_pos} and the fact that a smooth conic has $q+1$ rational points over $\FF_q$.
\end{proof}

Suppose now that $q\geq 4$. Let $P_1,\dots,P_5\in\PP^2(\FF_q)$ be in general position and let $C$ be the smooth conic passing through them. 
We next compute the number of rational points on the union $\mathfrak{C}$ of $C$ and the ten lines through $P_1,\dots,P_5$.

\begin{lemma}\label{lem:a=7}
For $q\geq 4$, we have $\#\mathfrak{C}(\FF_q) = 11q - 24$.
\end{lemma}

\begin{proof}
Let $\mathfrak{L}$ denote the union of the ten lines determined by $P_1,\dots,P_5$. Since the number of points on a line of $\mathfrak{L}$ where four lines meet is $2$ and the number of points on a line of $\mathfrak{L}$ where only two lines meet is $3$, the number of rational points of $\mathfrak{L}$ is 
$$
\#\mathfrak{L}(\FF_q)=10(q+1-5)+\frac{10\cdot 2}{4}+\frac{10\cdot 3}{2}.
$$
Adding this to $\#(C\setminus \mathfrak{L})(\FF_q) = q-4$ gives the result.
\end{proof}

\begin{corollary}\label{corollary: six points} $\PP^2_{\FF_q}$ has $6$ rational points in general position if and only if 
$q \neq 2,3,5$.
\end{corollary}

\begin{proof}
By Lemma \ref{lem:a=7}, we have $\#(\PP^2 \setminus \mathfrak C)(\FF_q) = (q - 5)^2$. For $q\geq 4$, this is strictly positive if and only if $q \neq 5$.
\end{proof}
This completes the case $a=7$. Let now $q$ be an arbitrary prime power and fix a smooth conic $C$ over $\FF_q$.

\subsubsection*{$a=5$}
The conic $C$ contains three rational points $P_1,P_2,P_3$ and a closed point of degree $2$; these lie in general position by Lemma \ref{lem:general_pos}. Denote by $L_1$, $L_2,$ $L_3$ the lines through each of the pairs of $P_1,P_2,P_3$. Denote by $P_4$ and $P_5$ the points over $\FF_{q^2}$ determined by the closed point and $L$ the line through them.
The line connecting $P_i$ with $P_j$, for $i\in \{1,2,3\}$ and $j\in\{4,5\}$ 
has only $P_i$ as a rational point.  Let $\mathfrak{L}=L_1\cup L_2\cup L_3 \cup L$ and $\mathfrak{C}= \mathfrak{L}\cup C$. From the above considerations it follows that
\begin{equation} \label{eqn:a=5}
\# \mathfrak{C}(\FF_q)=\# \mathfrak{L}(\FF_q)+q-2=4(q+1)-6+q-2=5q-4.
\end{equation}
As $\#(\PP^2 \setminus \mathfrak{C})(\FF_q) = (q-2)^2 + 1 > 0$, we find that there is a sixth point $P_6$ such that
$P_1,\dots,P_6$ are in general position, as required.

\subsection*{$a=4$}
Choose $P_1 \in C(\FF_{q^3}) \setminus C(\FF_q)$. Denote by $P_2$ and $P_3$ the conjugates of $P_1$ and choose $P_4,P_5 \in C(\FF_q)$. By Lemma \ref{lem:general_pos} these points lie in general position. Let $L$ be the line through $P_4$ and $P_5$.  There exists $P_6 \in \PP^2(\FF_q)$ such that $P_1,\dots,P_6$ are in general position as
$$
q^2+q+1-\#(L \cup C)(\FF_q)=(q-1)^2+q > 0.
$$

\subsection*{$a=3$} Let $P_1 \in C(\FF_{q^4}) \setminus C(\FF_{q^2})$. Denote by $P_2,P_3,P_4$ the conjugates of $P_1$ and let $P_5 \in C(\FF_q)$; these lie in general position by Lemma \ref{lem:general_pos}.
The lines through $P_1, P_3$ and $P_2,P_4$ meet in a rational point, which we call $Q$. One easily sees that $Q$ and $P_5$ are the only rational points which lie on some line through any two of the $P_i$. We conclude as before by noting that $q^2+q+1-\#(C\cup Q)(\FF_q)=q^2-1 >0$.

\subsection*{$a=2$} Choose a point in $C(\FF_{q^5}) \setminus C(\FF_{q})$. The conjugates $P_i$ of this point give a closed point $P$ of degree $5$ on $C$, which lies in general position by Lemma \ref{lem:general_pos}. One easily sees that lines through pairs of the $P_i$ contain no $\FF_q$-points. We conclude as before by noting that $q^2+q+1-\# C(\FF_q)>0$.

\subsection*{$a=1$} Let $\alpha_1,\dots,\alpha_6$ be a normal basis of $\FF_{q^6}$ over $\FF_q$, i.e.~the $\alpha_i$ are a basis of $\FF_{q^6}/\FF_q$ with  $\alpha_{i+1}=\alpha_i^{q}$ for $i=1,\dots, 6$, where the subscripts are taken modulo $6$. Write $P_i=[1:\alpha_i:\alpha^3_i]$ and note that the collection $\{P_1,\dots, P_6\}$ forms a closed point of degree $6$. For distinct $P_i$, $P_j$, and $P_k$, the determinant \eqref{mat:collinear} here is 
\begin{equation*}
(\alpha_k-\alpha_i)(\alpha_k-\alpha_i)(\alpha_j-\alpha_i)(\alpha_i+\alpha_j+\alpha_k)\neq 0\,,
\end{equation*}
hence $P_i,P_j,P_k$ are not collinear by Lemma \ref{lem:general_pos_det}. The matrix \eqref{mat:conic} has determinant

\begin{equation*}
\prod_{1\leq i< j\leq 6}(\alpha_j-\alpha_i)(\alpha_1+\dots+\alpha_6) \neq 0 \,,
\end{equation*}
hence the points lie in general position by Lemma \ref{lem:general_pos_det}.

\subsection*{$a=0$} We will show the existence of two closed points of degree 2 and one closed point of degree 3 in general position (we choose this configuration as it will also be used in the proof of Theorem \ref{thm:dp4}). Blowing these up we obtain a degree $2$ del Pezzo surface with trace $1$. This surface contains a line $L$,  corresponding to the line in $\PP^2$ passing through one of the closed points of degree 2. By Lemma~\ref{lem:blow_up}, the blow-down of $L$ yields the required cubic surface $S$ with $a(S) = 0$ (this is No.~14 in Table \ref{tab:Manin}).

Choose $\bar{\FF}_q$-points $P_1,P_2$ and $Q_1,Q_2,Q_3$ on $C$ which form closed points of degree $2$ and $3$, respectively; these lie in general position by Lemma \ref{lem:general_pos}. Let $L$ be the line through $P_1,P_2$. As $\#\PP^2(\FF_{q^2}) - \#\PP^2(\FF_{q}) - \#C(\FF_{q^2}) - \#L(\FF_{q^2}) =q^4 - 2q^2 - q - 2 > 0$, there are $R_1,R_2 \in \PP^2(\FF_{q^2})$ that form a closed point of degree $2$ not in $C \cup L$. Moreover, a simple consideration of the Galois action shows that no $3$  are collinear.

It remains to show that no 6 lie on a conic. Let $D$ be a conic passing through 6 of the points and let $D^{\Frob_q}$ be the conjugate conic under the action of $\Frob_q \in \Gal(\bar{\FF}_q/\FF_q)$. The conics $D$ and $D^{\Frob_q}$ have $5$ points in common, thus $D=D^{\Frob_q}$, i.e.~$D$ is defined over $\FF_q$. As none of our points are rational, we see that $D$ actually contains all seven points, hence $D = C$. But $R_i \notin C$ by construction; a contradiction.

\subsection*{$a=-1$} Such a surface was constructed by Rybakov in the proof of \cite[Thm.~3.2]{Ryb05} for his study of quartic del Pezzo surfaces. For completeness we recall this construction and clarify why it works for all $q$.

Recall from class field theory \cite[Thm.~1.5.36(i)]{Poo17} that we have the following fundamental
exact sequence 
\begin{equation} \label{seq:BR}
	0 \longrightarrow \Br K \longrightarrow \bigoplus_{P \in \PP^1} \Br K_P \overset{\sum_P \inv_P}{\longrightarrow} \QQ/\ZZ \longrightarrow 0
\end{equation}
of the Brauer group of $K=\FF_q(t)$. Here the direct sum is over the closed points of $\PP^1_{\FF_q}$, and $K_P$ denotes the completion of $K$ at $P$. (Section 1.5 of Poonen's book \cite{Poo17} is a great reference for the Brauer group facts we shall require).

Let $P_1,P_2,P_3 \in \PP^1(\FF_q)$ be distinct and take $P_4 \in \PP^1$ a closed point of degree $2$ (these exist for all $q$).
From the exactness of \eqref{seq:BR}, there exists a $2$-torsion element $\alpha \in \Br K$ such that for all closed points $P \in \PP^1$ we have $\inv_P \neq 0$ if and only if $P \in \{P_1,\dots,P_4\}$. Moreover, by \cite[Thm.~1.5.36(iv)]{Poo17} the element $\alpha$ has index $2$, thus $\alpha$  has a representative which is a quaternion algebra, i.e.~a $4$-dimensional central simple algebra over $K$. To this quaternion algebra one may naturally associate a conic $C$ over $K$ \cite[Prop.~1.5.9]{Poo17}. By the theory of minimal models of surfaces (see \cite{Isk79}, or in particular \cite[Lem.~2.8]{Ryb05}) there exists a  relatively minimal conic bundle $\pi:S \to \PP^1$ whose generic fibre is isomorphic to $C$. Moreover, this has the property that the fibre over a closed point $P$ is singular if and only if $P \in \{P_1,\dots,P_4\}$ \cite[Lem.~2.9]{Ryb05}. We claim that this is the required surface.

To see this note that $\pi:S \to \PP^1$ is a relatively minimal conic bundle with exactly $5$ singular fibres over $\bar{\FF}_q$. It thus follows from \cite[Thm.~3(3)]{Isk79} and \cite[Thm.~5(1)]{Isk79} that $S$ is in fact a cubic surface. Finally Lemma \ref{lem:conic_bundle} implies that $a(S) = -1$, as claimed (this yields No.~9 from Table~\ref{tab:Manin}).

\subsection*{$a=-2$} This case was already handled by Swinnerton-Dyer \cite{SD10}.

\smallskip
\noindent This completes the proof of Theorem \ref{thm:dp3}. 

\subsection{Proof of Theorem \ref{thm:dp4}}
The case $a=6$ follows from Lemma \ref{lem:five_points}. For $1 \leq  a \leq 4$, in the proof of Theorem \ref{thm:dp3} we showed the existence of a smooth cubic surface $S$ with a line over all $\FF_q$ such that $a(S) = a+1$; contracting this line yields the required quartic del Pezzo surface by Lemma \ref{lem:blow_up}. For $a=-1$, consider the points $P_i,Q_i,R_i$ constructed in the trace $0$ case of cubic surfaces and let $S'$ be their blow-up. This contains two skew lines, corresponding to the line through $P_1,P_2$ and the conic through $P_1,P_2,Q_1,Q_2,Q_3$. The blow down of these lines yields the required surface. For $a=-2,0$, the required surfaces have been constructed by Rybakov \cite[Thm.~3.2]{Ryb05} (these are $X$ and $XVIII$ from Table \ref{tab:Manin}, respectively).

\section{Del Pezzo surfaces of degree \texorpdfstring{$2$}{Lg}}\label{sec:DP2}
In this section we establish Theorem \ref{thm:dp2}.

\subsection{Definitions and basic properties} \label{sec:Def_DP2}

Let $k$ be a field. Any del Pezzo surface $S$ of degree $2$ over $k$ can be written
in the form
\begin{equation} \label{eqn:DP2}
	w^2 + f_2(x,y,z)w = f_4(x,y,z) \quad \subset \PP(1,1,1,2), \qquad \deg f_i = i. 
\end{equation}

\subsubsection{The ramification curve}
The anticanonical map $\pi: S \to \PP^2$ is given by
$[x:y:z:w] \mapsto [x:y:z]$,
and realises $S$ as a double cover of $\PP^2$. 
The behaviour in characteristic $2$ is slightly different; a good general
reference for double covers in characteristic $2$ is \cite[\S0.1]{CD89}.
The morphism $\pi$ is separable in all characteristics.

When $\chr k \neq 2$, we may choose equations so that $f_2(x,y,z)=0$. 
In which case, the double cover is ramified over the smooth quartic curve $B: f_4(x,y,z) = 0$.
When $\chr k = 2$ the branch curve $B$ is the plane conic $f_2(x,y,z) = 0$ (this can be reducible
or non-reduced). In both cases, we define the ramification curve to be $R = \pi^{-1}(B)_{red}$, 
i.e.~the reduced subscheme underlying $\pi^{-1}(B)$.

The following lemma on the geometry of the ramification curve will be used in Section~\ref{sec:DP1}.
We define the genus $g(C)$ of a geometrically irreducible (possibly singular or non-reduced) projective curve $C$
to be the genus of the normalisation of $C_{red}$.

\begin{lemma} \label{lem:ramification_curve}
	Let $S$ be a del Pezzo surface of degree $2$ over an algebraically closed field $k$ with ramification curve $R$.
	\begin{itemize}
		\item If $\chr k \neq 2$ then $R$ is irreducible smooth and of genus $3$.
		\item If $\chr k = 2$ then $R$  has at most $2$ irreducible components, and each irreducible 
		component has genus $0$.
	\end{itemize}
\end{lemma}

\begin{proof}
	When $\chr k \neq 2$ the result is clear, as $R \cong B$ is a smooth plane quartic. So assume that $\chr k = 2$.
	Here $\pi^{-1}(B)$ has the equation
	$$\pi^{-1}(B): \quad f_2(x,y,z) = 0, \,\,  w^2 = f_4(x,y,z) \quad \subset S.$$
	It suffices to consider the various possibilities for $B$.
	\begin{enumerate}
		\item $B$ is a smooth plane conic: Here $R = \pi^{-1}(B)$ is irreducible and reduced, but may be singular. The morphism $R \to B$
		is purely inseparable of degree $2$. Let $N \to R$ be the normalisation of $R$. The induced map $N \to B$ is still purely
		inseparable of degree $2$, hence $g(N) = g(B)$ (see e.g.~\cite[Lem.~8.6.6]{Kem93}) and thus $g(R)=0$.
		\item $B=L_1 \cup L_2$ is a union of $2$ distinct lines: 
		Each $R_i:=\pi^{-1}(L_i)$ is irreducible and reduced and the
		map $R_i \to L_i$ is purely inseparable of degree $2$. As in the previous case, we find that $g(R_i) = 0$.
		\item $B = L^2$ is a double line: Here $\pi^{-1}(B)$ is non-reduced, but $R \to L$ is still purely inseparable of degree $2$.
		As above, we conclude that $g(R) = 0$. 
	\end{enumerate}
\end{proof}

\subsubsection{Geiser twists} \label{sec:quadratic_twist_2}
The map $\pi$ induces an involution of $S$, called the \emph{Geiser involution}.
We may therefore twist by some cocycle $\alpha$ with class in $\mathrm{H}^1(k,\ZZ/2\ZZ)$, where $\ZZ/2\ZZ$ acts on $S$ via the Geiser involution, to obtain the \emph{Geiser twist}
$S_\alpha$ by $\alpha$ (this is a quadratic twist of $S$).
Over a finite field there is a unique non-trivial Geiser
twist up to isomorphism; we let $S_{\sigma}$ denote the choice of such a twist.

For completeness we give equations for these twists, though these will not be used in the sequel.
When $\chr k \neq 2$, we choose the equation \eqref{eqn:DP2} so that $f_2 = 0$. 
Kummer theory gives $\mathrm{H}^1(k,\ZZ/2\ZZ)=k^*/k^{*2}$, and for $\alpha \in k^*$
we have 
$$S_\alpha: \quad \alpha w^2  = f_4(x,y,z).$$
When $\chr k = 2$, as $\pi$ is separable, one may write down equations for $S_\alpha$ using Artin-Schreier theory
instead of Kummer theory. This yields $\mathrm{H}^1(k, \ZZ/2\ZZ) \cong k/\wp k$,
where $\wp(\alpha) = \alpha^2 - \alpha$. For $\alpha \in k$, the associated Geiser twist
is given by
$$S_\alpha: \quad w^2 + f_2(x,y,z)w = f_4(x,y,z) + \alpha f_2(x,y,z)^2.$$

\begin{lemma} \label{lem:quadratic_2}
	Let $S$ be a del Pezzo surface of degree $2$ over a finite field $\FF_q$ and 
	$S_\sigma$ its non-trivial Geiser twist. Then
	$a(S) + a(S_\sigma) = 2.$
\end{lemma}
\begin{proof}
	Let $\pi:S \to \PP^2$ (resp.~$\pi_\sigma:S_\sigma \to \PP^2$) be the associated
	double cover of $\PP^2$, with branch locus $B$. Let $x \in \PP^2(\FF_q)$. If $x \in B$,
	then $\pi^{-1}(x)$ and $\pi_\sigma^{-1}(x)$ both have a single rational point
	(if $\chr k \neq 2$ this is clear; if $\chr k = 2$ then one observes,
	as in the proof of Lemma \ref{lem:ramification_curve}, that $R_{red} \to B_{red}$ is purely inseparable).
	If $x \notin B$, then one of $\pi^{-1}(x)$ or
	$\pi_\sigma^{-1}(x)$ contains exactly two rational points, and the other none. 
	Taking these contributions together we obtain
	$$\#S(\FF_q) + \#S_\sigma(\FF_q) = 2\#\PP^2(\FF_q) = 2(q^2 + q + 1).$$
	The result follows on recalling \eqref{eqn:Weil}.
\end{proof}

\begin{remark}
Analogues of Lemma \ref{lem:quadratic_2} for elliptic curves are well-known; see for example Exercises 61 and 62 of \cite{Eke06}.
\end{remark}

\subsection{Proof of Theorem \ref{thm:dp2}}
By Lemma \ref{lem:quadratic_2} we need only consider $a \geq 1$. We shall handle these cases using a similar strategy to the proof of Theorem \ref{thm:dp3}.

\subsection*{$a=8$}Here we are concerned with characterizing those $q$ for which there exist $7$ points in $\PP^2(\FF_q)$ in general position. There are already some results in the literature concerning this problem. For example in \cite[Lem.~68]{Kap13} this problem is solved for all \emph{odd} $q$. We give a new proof which applies to \emph{all} $q$.

\begin{proposition} \label{prop:7_points}
$\PP^2_{\FF_q}$ has  $7$ rational points in general position if and only if $q \geq 9$. 
\end{proposition}
\begin{proof}
Let $S$ be the degree $2$ del Pezzo surface given by blowing up these points. The blow-down of a line $L$ is a split smooth cubic surface $S'$, and the image of $L$ is a rational point not on a line. However, as proved by Hirschfeld in \cite[Thm.~20.3.9, Thm.~20.3.10]{Hir85} (see also \cite{Hir82}),  there is no such smooth cubic surface when $q \leq 8$.

Suppose now that $q\geq 9$. In \cite{Hir82}, it is shown that there is a smooth split cubic surface over $\FF_q$ which contains a rational point not on a line. We give our own proof, as the paper \cite{Hir82} is difficult to obtain. Let $P_1,\dots,P_6$ be $6$ rational points in general position and let $\mathfrak{L}$ denote the union of the $15$ lines through pairs of them. Let $n_i$ be the number of points where exactly $i$ of these lines meet. One sees that
$$ \binom{5}{2}n_5 + \binom{3}{2}n_3 + n_2 = \binom{15}{2}.$$
As $n_5 = 6$, we deduce that $3n_3 + n_2 = 45$ and that
$$
\#\mathfrak{L}(\FF_q)=15(q+1)-4 n_5-2n_3-n_2=15q-54+n_3.
$$
Let $\mathfrak C$ denote the union of $\mathfrak L$ and the six conics through any five of the points. Each conic adds at most $q-4$ points to the configuration. Hence 
$$
\#\mathfrak{C}(\FF_q)\leq 15q-54+n_3 + 6(q-4) = 21q-78+n_3\leq 21 q-63.
$$
We deduce that 
$$
\#(\PP^2 \setminus \mathfrak C)(\FF_q) \geq (q-10)^2-36.
$$
This is positive for $q \geq 17$. For the remaining values $q=9,11,13,16$, rational points in general position are easily found using a computer search and Lemma~\ref{lem:general_pos_det}.
\end{proof}

\subsection*{$a=6$} By \cite[Tab.~1]{Ura96a}, any degree 2  del Pezzo surface $S$  with $a(S) = 6$ is the blow-up of $\PP^2$ in  $5$ rational points and a closed point of degree $2$ in general position. 

By Lemma \ref{lem:five_points}, such a configuration does not exist for $q=2,3$. The value $q=4$ is small enough that one can quickly enumerate the relevant collections of points in $\PP^2$ on a computer, and verify their non-existence using Lemma \ref{lem:general_pos_det}.

Let now  $q \geq 5$. Let $P_1,P_2,P_3 \in \PP^2(\FF_q)$ and choose $P_4,P_5 \in \PP^2(\FF_{q^2})$ which form a closed point of degree $2$, such that $P_1,P_2,P_3,P_4,P_5$ lie in general position (the existence of such a collection was shown in the proof of Theorem \ref{thm:dp3}). Let $\mathfrak{C}$ be the union of the conic through all five and the ten lines through pairs of these five points. By \eqref{eqn:a=5} we have $\#\mathfrak{C}(\FF_q) = 5q - 4$. Choose a sixth rational point $P_6$ which lies in general position with respect to the others.

Let $\mathfrak{C}'$ be the union of $\mathfrak{C}$ with the five lines from $P_6$ to the other five points and the five conics through $P_6$ and four of the $P_1,\ldots,P_5$ (we call these the \emph{five new lines} and \emph{five new conics} respectively, and the rational points on them \emph{not} on $\mathfrak{C}$ nor equal to $P_6$ as \emph{new rational points}). Only three of the five new lines are rational, and the number of new rational points on the new lines is at most $3q - 6$. Similarly, the number of new rational points on the new conics is at most $3q - 6$. Thus 
$$\#\mathfrak{C}'(\FF_q) \leq \#\mathfrak{C}(\FF_q) + 1 + 2(3q -6) = 11q - 15.$$
Hence $\#(\PP^2 \setminus \mathfrak{C}')(\FF_q) \geq (q-5)^2 - 9$, and so we can choose a seventh rational point in general position provided $q \geq 9$.

For the remaining values $q=5,7,8$ a computer search reveals that the required configurations exist.

\subsection*{$a=5$} By \cite[Tab.~1]{Ura96a}, any degree $2$ del Pezzo surface $S$  with $a(S) = 5$ is the blow-up of $\PP^2$ in $4$ rational points and a closed point of degree $3$ in general position. Such a surface $S$ does not exist over $\FF_2$: by Lemma \ref{lem:quadratic_2} its Geiser twist $S_\sigma$ has $a(S_\sigma)=-3$, hence $S_\sigma$ has a negative number of rational points by \eqref{eqn:Weil}.

Let $q\geq 3$ and let $C$ be a smooth conic over $\FF_q$. Let $P_1 \in C(\FF_{q^3}) \setminus C(\FF_q)$ with conjugates $P_2,P_3$. Choose $P_4,P_5 \in C(\FF_q)$ and let $P_6 \in \PP^2(\FF_q)$ be in general position with respect to the other five; this was shown to exist in the proof of Theorem \ref{thm:dp3}. Let $\mathfrak{C}$ be the union of $C$ and the line through $P_4,P_5$. We have $\#\mathfrak{C}(\FF_q) = 2q$.

Now consider the 5 new lines and 5 new conics, as in the $a = 6$ case. Only two of the new lines and two of the new conics are rational, and the non-rational lines and conics contain no new rational points. Thus the number of new rational points is at most $4(q - 1)$.
We have $q^2 + q + 1 - 1 - 4(q - 1) - 2q = q^2 - 5q + 4$, thus for $q \geq 5$ one may always choose a seventh rational point in general position.

For $q = 3,4$ a computer search reveals that such collections of points exist.

\subsection*{$a=4$} We first rule out $q = 2$. Note that, by Lemma \ref{lem:quadratic_2}, the Geiser twist of such a surface over $\FF_2$ has a unique rational point. The non-existence of such a surface has been proved in \cite[p.~14]{STVA14} and \cite[Thm.~4.1.1]{Li10}. Both proofs used intensive computer searches to enumerate del Pezzo surfaces over $\FF_2$ and verify that no such surface exists. We give a more conceptual proof,  using the classification of conjugacy classes due to Urabe \cite{Ura96a} (though of course some computation is implicitly used in the proof, as a computer was used to determine the conjugacy classes of $W(\E_7)$).

\begin{lemma}
	There is no del Pezzo surface $S$ of degree $2$ over $\FF_2$ with $a(S) = 4$.
\end{lemma}
\begin{proof}
	An inspection of \cite[Tab.~1]{Ura96a} reveals that such a surface must have number $27,48$ or $52$, in the notation
	of \emph{loc.~cit.}. 
	
	Class $48$ has index $7$ and arises by blowing up $\PP^2$ in $3$ rational points and
	$2$ closed points of degree $2$. This configuration cannot arise over $\FF_2$;
	indeed, after a quadratic extension, one would obtain $7$ rational points
	in general position over $\FF_4$, which do not exist by Proposition \ref{prop:7_points}.
	
	Class $52$ also has index $7$ and arises by blowing up $\PP^2$ in $3$ rational points
	$P,Q,R$ and a closed point of degree $4$. Let $P_1,P_2,P_3,P_4$ be the points over $\bar{\FF}_q$
	given by the closed point of degree $4$, chosen so that the Frobenius
	element $\Frob_q$ acts via 	$\Frob_q(P_i) = \Frob_q(P_{i+1})$,
	where the subscripts are taken modulo $4$. Consider the pencil of conics in $\PP^2$ which 
	pass through $P_1,\dots,P_4$. This pencil contains $4$ elements over $\FF_q$:
	the $3$ conics through $P_1,\dots,P_4$ and one of the points $P,Q,R$, as well
	as the singular conic given by the lines passing through $P_1,P_3$ and $P_2,P_4$.
	However this is a contradiction as $\# \PP^1(\FF_2) = 3 < 4$.
	
	Consider next a surface $S$ with class $27$, which has index $2$ and contains $2$ skew lines.
	By Table \ref{tab:Manin}, the blow-down of these lines is a minimal quartic del Pezzo surface $S'$
	with Picard number $2$. Hence by \cite[Thm.~1]{Isk79}, the surface $S'$ is equipped with a relatively
	minimal conic bundle $S' \to \PP^1$ which has $4$ singular fibres over $\bar{\FF}_q$. As $a(S') = 2$, Lemma~\ref{lem:conic_bundle} implies that
	the singular fibres  lie over $2$ distinct closed points of $\PP^1$ of degree $2$. However $\PP^1$
	contains only $1$ closed point of degree $2$ over ${\FF_2}$, which is a contradiction.
\end{proof}

Assume now that $q \geq 3$. We claim that there are three rational points and a closed point of degree $4$ in general position. A similar method to the previous cases shows that there are at most $3q + 2$ rational points on the $15$ lines and $6$ conics through two rational points and a closed point of degree 4 in general position. So for $q \geq 3$ there is a seventh rational point in general position.

\subsection*{$a=3$} Consider a closed point of degree 5 on a smooth conic $C$ and a rational point $P \notin C$. The only rational point of $\PP^2 \setminus C$ which lies on one of the relevant lines and conics is $P$, and thus the resulting configuration of lines and conics contains exactly $q + 2$ rational points; hence there is a seventh rational point in general position.

\subsection*{$a=2$}Consider the closed point of degree 6 in general position that was constructed in the proof of Theorem \ref{thm:dp3}, and consider the 15 lines and 6 conics that it defines. This configuration has at most one rational point (three of the lines might cross in one point) and so one can always choose a rational point not on this configuration. 

\subsection*{$a=1$}  Let $\alpha_1,\dots,\alpha_7$ be a normal basis of $\FF_{q^7}$ over $\FF_q$. As in the proof of Theorem~\ref{thm:dp3}, we apply Lemma \ref{lem:general_pos_det} and find that the $P_i=[1:\alpha_i:\alpha^3_i]$ form a closed point of degree $7$ in general position.

\smallskip

\noindent
Applying Lemma \ref{lem:quadratic_2}, this completes the proof of Theorem \ref{thm:dp2}.

\section{Del Pezzo surfaces of degree \texorpdfstring{$1$}{Lg}}\label{sec:DP1}
In this section we prove Theorem \ref{thm:dp1}. 
	
\subsection{Definitions and basic properties}

Let $k$ be a field. Any del Pezzo surface $S$ of degree $1$ over $k$ can be written
in the form
$$w^2 + f_1(x,y)zw + f_3(x,y)w = z^3 + f_2(x,y)z^2 + f_4(x,y)z + f_6(x,y) \subset \PP(1,1,2,3)$$
where $\deg f_i = i$. The double of the anticanonical map is given by $[x:y:z:w] \mapsto [x:y:z]$,
and realises $S$ as a double cover of $\PP(1,1,2)$. 
This map is separable and induces an involution of $S$, called the \emph{Bertini involution}.

\subsubsection{Bertini twists}
As in \S\ref{sec:quadratic_twist_2}, we may perform a quadratic twist by an element of $\alpha \in \mathrm{H}^1(k, \ZZ/2\ZZ)$ to obtain the \emph{Bertini twist} $S_\alpha$ of $S$.
Over a finite field, we denote the unique non-trivial Bertini twist of $S$ by $S_\sigma$.
We have the following analogue of Lemma \ref{lem:quadratic_2}.

\begin{lemma} \label{lem:quadratic_1}
	Let $S$ be a del Pezzo surface of degree $1$ over a finite field $\FF_q$ and
	$S_\sigma$ its non-trivial Bertini twist. Then $a(S) + a(S_\sigma) = 2.$
\end{lemma}
\begin{proof}
	A similar strategy to the proof of Lemma \ref{lem:quadratic_2} shows that
	$$\#S(\FF_q) + \#S_\sigma(\FF_q) = 2\#\PP(1,1,2)(\FF_q).$$
	However $\PP(1,1,2)$ has $1 + q + q^2$ rational points,
	and the result follows.
\end{proof}

\subsection{Proof of Theorem \ref{thm:dp1}}\label{subsec:proofdp1}
We follow a similar strategy to the previous cases. However, to avoid working directly with the condition in Definition \ref{def:gen_pos} concerning cubic curves passing through the $8$ points, we shall often take the following approach.

Let $S'$ be a degree $2$ del Pezzo surface. The blow-up $S$ of $S'$ in a rational point $P$ is a 
 del Pezzo surface if and only only if 
$P$ does not lie on a line nor on the ramification curve \cite[Cor.~14]{STVA14} (we call the union of the lines and the ramification curve the \emph{bad locus}). We can use our constructions 
for degree $2$ del Pezzo surfaces to obtain certain traces, and then blow-up such a rational point to obtain the required surface of degree $1$. We will deduce the existence of a rational point not in the bad locus for sufficiently large $q$ using Lemma~\ref{lem:ramification_curve} and the Hasse-Weil bounds, which imply that the ramification divisor has at most $F(q)$ rational points, where
\begin{equation} \label{def:F}
	F(q) = 
	\begin{cases}
		q+1+6\sqrt{q}, & q \text{ odd}, \\
		2(q+1), & q \text{ even}.
	\end{cases}
\end{equation}

\subsubsection*{$a=9$} Every such degree $1$ del Pezzo surface is split.

\begin{lemma}
	Let $q \leq 13$ or $q=17$. Then there does not exist a split del Pezzo surface of degree
	$1$ over $\FF_q$.
\end{lemma}
\begin{proof}
	Let $S$ be such a  surface. The blow-down of a line on $S$ is a split del
	Pezzo surface $S'$ of degree $2$ with a rational point outside the bad locus.
	Hence it suffices to show that there exists no such surface $S'$,
	for $q$ as in the statement of the lemma.

	Theorem \ref{thm:dp2} rules out $q \leq 9$.
	Kaplan \cite[\S4.3]{Kap13} has determined the isomorphism classes of 
	split del Pezzo surfaces of degree $2$ over the remaining relevant
	finite fields. As observed in \cite[Cor.~4.4]{KR16}, this classification implies that for
	any split del Pezzo surface of degree $2$ with $q=9,11,13$, all its rational
	points lie on its lines (i.e.~is ``full'' in the terminology of \emph{loc.~cit.}).
	By \cite[Prop.~73]{Kap13} there are $7$ isomorphism
	classes of split del Pezzo surfaces of degree $2$ over $\FF_{17}$, and by
	\cite[Thm.~4.5]{KR16} all but one of these is full. As explained in the proof of
	\cite[Thm.~4.5]{KR16}, the non-full surface $S'$
	is branched over the Kuwata quartic curve $C_{234}$
	(see \cite[\S3]{KR16} and \cite[\S8]{Kuw05} for notation). One finds that
	$S'$ has the equation
	$$w^2 = x^4 + 13x^2y^2 + 13y^4 + 13x^2z^2 + 4y^2z^2 + 8z^4.$$
	Using the explicit description of the lines given in \cite[Thm.~3.2]{KR16}
	(see also \cite[Thm.~8.2]{Kuw05}), a calculation shows that $S'$ contains exactly $2$ rational
	points which do not lie on a line. These are the points $[0:6:1:0]$ and 
	$[0:11:1:0]$, which clearly lie on the ramification locus. This completes
	the proof.
\end{proof}

Let $q \geq 16$ and let $S'$ be a split degree 2 del Pezzo surface over $\FF_q$.
It has $56$ lines over $\FF_q$.  
Thus the number of points in the bad locus is at most $56(q+1) + F(q)$, where $F(q)$ is given by \eqref{def:F}. Comparing this with $\#S'(\FF_q) = q^2 + 8q + 1$, we find that there is a rational point outside the bad locus provided $q \geq 53$.

For $q=16$ and $19 \leq q \leq 49$, after a computer search we find $8$ rational points in general position in $\PP^2$.

\subsubsection*{$a=7$} By \cite[Tab.~2]{Ura96a}, every such degree $1$ del Pezzo surface is the blow-up of $\PP^2$ in $6$ rational points and a closed point of degree $2$ in general position. It follows from Theorem \ref{thm:dp3} and Theorem \ref{thm:dp2}, that these surfaces do not exist for $q\leq 5$. A computer search enumerating such collections of points rules out the remaining values $q=7,8,9$.

Let $q \geq 11$ and let $S'$ be a degree 2 del Pezzo surface over $\FF_q$ given as the blow-up of 5 rational points and one closed point of degree 2 in general position, as constructed in the proof of Theorem \ref{thm:dp2}. Precisely 32 of the 56 lines of $S'$ are rational. Moreover any rational point on one of the 24 non-rational lines must lie on one of the $5$  exceptional lines. Therefore, there are at most $32(q+1) + F(q)$ rational points in the bad locus. Comparing this with $\#S'(\FF_q) = q^2 + 6q + 1$, we see that $S'$ admits a rational point away from the bad locus when $q \geq 31$. 

For $11 \leq q \leq 29$ we find six rational points and one closed point of degree $2$ lying in the plane in general position, after a computer search.

\subsubsection*{$a=6$} By \cite[Tab.~2]{Ura96a}, every such degree $1$ del Pezzo surface $S$ is the blow-up of $\PP^2$ in $5$ rational points and a closed point of degree $3$ in general position. Such a surface does not exist for $q=2,3$ by Lemma \ref{lem:five_points}. For $q=4,5$ the required configuration of points in $\PP^2$ is shown not to exist after a computer search.

Let $q \geq 7$ and let $S'$ be the degree 2 del Pezzo surface over $\FF_q$ obtained by blowing-up $\PP^2$ in 4 rational points and one closed point of degree $3$. This surface has precisely 20 of its lines rational, and any rational point on a non-rational line must also lie on some rational line; thus the bad locus contains at most $20(q+1) + F(q)$ rational points. Comparing this with $q^2 + 5q + 1$ yields the required surface for $q \geq 19$.

For $7 \leq q \leq 17$, we find five rational points and a closed point of degree $3$ in general position after a computer search.

\subsubsection*{$a=5$} 

Such a surface $S$ cannot exist when $q=2$, as by Lemma~\ref{lem:quadratic_1} its Bertini twist would have $a(S_\sigma)=-3$, hence a negative number of rational points by \eqref{eqn:Weil}.

We now assume that $q \geq 3$. There are three rational points and a closed point of degree 4 in $\PP^2$ in general position. The corresponding degree 2 del Pezzo surface $S'$ has precisely 12 of its lines rational. Moreover, there are two pairs of non-rational lines on $S'$ which each meet in a rational point: these correspond to the $2$ lines (resp.~$2$ conics) in $\PP^2$ which pass through non-adjacent quartic points. Thus $2 + 12(q+1) + F(q)$ is an upper bound for the number of rational points in the bad locus, which is less than $q^2 + 4q + 1$ provided $q \geq 13$. 

A computer search shows the existence of four rational points and a closed point of degree $4$ in general position for $5 \leq q \leq 13$, but also reveals that there is no surface $S$ of index $8$ with $a(S) = 5$ over $\FF_q$ for $q \in \{2,3,4\}$. For $q=3,4$ we therefore explicitly find such a surface, necessarily not of index $8$.

For $q=3$ there is a unique such surface up to isomorphism, as discovered by Li \cite[Thm.~3.1.3]{Li10}. It has the equation
$$w^2 = z^3 + (2x^4 + x^2y^2 + 2y^4)z + (x^6 + 2x^4y^2 + y^6).$$
For $q=4$ and $u \in \FF_4\setminus\FF_2$, an example of a surface of trace $5$ is 
\begin{align*}
	& w^2 + xzw + (x^3 + u^2x^2y + uxy^2 + y^3)w = \\
	& \quad z^3 + (u^2x^2 + xy + y^2)x^2z  +  (u^2x^4 + ux^3y  + x^2y^2)y^2.
\end{align*}     

\subsubsection*{$a=4$} For all $q$ we may find two rational points and a closed point of degree 5 in $\PP^2$ in general position. The corresponding degree 2 del Pezzo surface $S'$ has precisely 6 rational lines, and none of its non-rational lines contains a rational point not already on some rational line. Thus there are at most $6(q+1) + F(q)$ rational points in the bad locus, which is less than $q^2 + 3q + 1$ provided $q\geq 8$.
A computer search finds three rational points and a closed point of degree $5$ in general position for $q = 3,4,5,7$, but shows that no index $8$ surfaces occur here for $q=2$.

Surfaces over $\FF_2$ with a unique rational point were found by Li \cite[Thm.~3.1.3]{Li10}.
The Bertini twist of such a surface has trace $4$, an example being
$$w^2 + (x^3 + x^2y + y^3)w = z^3  + (x^4 + x^3y + y^4)z.$$

\subsubsection*{$a=3$} We consider the degree $2$ del Pezzo surface $S'$  given by blowing up $\PP^2$ in a rational point and a closed point of degree $6$. There are $2$ rational lines, and at most $2$ rational points on the non-rational lines (these could occur if the $3$ lines through a pair of opposite sextic points are concurrent, or if the $3$ conics which avoid a pair of opposite sextic points meet in a point). Comparing $2(q+1) +2 + F(q)$ with $1 + 2q + q^2$, we find that there is a rational point away from the bad locus for $q \geq 4$. A computer search finds two rational points and a closed point of degree $6$ in general position when $q=3$, but shows that there are no surfaces of index $8$ with trace $3$ over $\FF_2$.
An explicit surface of trace $3$ over $\FF_2$ is given by
$$w^2  + yzw + (x^3 + xy^2)w = z^3 + (x^4 + xy^3)z + (x^5y + x^4y^2  + x^3y^3 + x^2y^4 + y^6) .$$

\subsubsection*{$a=2$} Let $S'$ be a del Pezzo surface over $\FF_q$ obtained as the blow-up of a closed point of degree $7$ (this was constructed in the proof of Theorem \ref{thm:dp2}). None of the 56 lines on this surface has any rational points. The ramification divisor, and hence the bad locus, has at most $F(q)$ rational points, which is less than the total number $q^2 + q + 1$ of rational points whenever $q=2$ or $q \geq 4$. A computer search finds one rational point and a closed point of degree $7$ in general position for $q=3$.

\subsubsection*{$a=1$} Let $\alpha_1,\dots,\alpha_8$ be a normal basis of $\FF_{q^8}$ over $\FF_q$. Write $P_i=[1:\alpha_i:\alpha^3_i]$ and note that $\{P_1,\dots, P_8\}$ forms a closed point. As in the proof of Theorem \ref{thm:dp3}, an application of Lemma \ref{lem:general_pos_det} shows that no $3$ are collinear and no $6$ lie on conic.

It remains to verify that the eight matrices $M_i$ from Part $(3)$ of Lemma \ref{lem:general_pos_det} all have trivial kernel. To do this we will show that, for $i=1,\dots,8$, the matrix given by removing the last row of $M_i$ has non-zero determinant.
This determinant is
$$
\pm \alpha_i^3 \left(2\alpha_i+\sum_{\substack{1\leq j \leq 8 \\ j \neq i }} \alpha_j\right)\prod_{\substack{1\leq j \leq 8 \\ j \neq i }}(\alpha_i-\alpha_j)^2\prod_{\substack{1\leq k< j\leq 8,\\ k,j\not = i}}(\alpha_k-\alpha_j)\,,
$$
and our claim follows from the fact that the $\alpha_i$ constitute a basis.

\smallskip

\noindent Applying Lemma \ref{lem:quadratic_1} completes the proof of Theorem \ref{thm:dp1}.

\section{An inverse Galois problem} \label{sec:inverse}
The aim of this section is to prove Theorem \ref{thm:inverse} and Theorem \ref{thm:inverse_2}.

\subsection{Hilbert schemes} Let $d\leq 6$.
We first count all \emph{anticanonically embedded} del Pezzo surfaces, viewed
as points of some Hilbert scheme, and show that an analogous limit exists. We then quotient out by the
action of the automorphism group.
We work with the set-up
of \cite[\S4]{JL15}.
Let
$$\mathcal{X}_d  = 
\begin{cases}
	\PP_\ZZ^d, & \quad  \mbox{if } d \geq 3, \\
	\PP(1,1,1,2)_\ZZ, & \quad  \mbox{if } d = 2, \\
	\PP(1,1,2,3)_\ZZ, & \quad  \mbox{if } d = 1.
\end{cases}
$$
Let $\mathcal{G}_d$ be the automorphism group scheme of $\mathcal{X}_d$ over $\ZZ$.
Any degree $d$ del Pezzo surface can be anticanonically embedded into $\mathcal{X}_d$. Let $\mathcal{H}_d$ denote
the Hilbert scheme of anticanonically embedded del Pezzo surfaces of degree $d$ over $\ZZ$. 
By \cite[Lem.~4.1]{JL15}, the morphism $\mathcal{H}_d \to \Spec \ZZ$ is smooth with geometrically connected fibres.

Let $\mathcal{L}_d \to \mathcal{H}_d$ be the universal family of lines of anticanonically embedded del Pezzo surfaces
of degree $d$ (see \cite[\S4.2]{JL15}).

\begin{proposition} \label{prop:etale}
	Let $d\leq 6$. The morphism $\mathcal{L}_d \to \mathcal{H}_d$ is finite \'{e}tale
	with irreducible generic fibre. Let $L$ and $H$ be the function fields of $\mathcal{L}_d$
	and $\mathcal{H}_d$, respectively, and $K$ the Galois closure of $L/H$. 
	Then $K \cap \bar{\QQ} = \QQ$ and 
	$\Gal(K/H) \cong W(\E_{9-d}).$
\end{proposition}
\begin{proof}
	For all number fields $k$, we first note the existence of a degree $d$
	del Pezzo surface over $k$ whose Hilbert scheme of lines
	is irreducible and whose splitting field has Galois group $W(\E_{9-d})$.
	For $d=5$ and $d=6$ this follows from the classification
	of such surfaces 
	(see e.g.~\cite[Thm.~3.1.3]{Sko01} and \cite[Thm.~3.5]{Blu10}, respectively).
	For $d=4$ this is proved in \cite[Thm.~I, p.~17]{KST89}. The case of 
	cubic surfaces has been known for a long time; 
	a modern proof can be found in \cite[Thm.~8.3]{Shi91}.
	For $d=1$, this follows from \cite[Thm.~1.3]{VAZ09}
	(see \cite[Rem.~1.4]{VAZ09}). The result for $d=2$ over $\QQ$ is proved in 
	\cite{Ern94}, and a similar argument to the one given at the
	end of \cite[\S6]{VAZ09}, which we do not reproduce here,
	yields the claim over any number field.
	
	We now turn to the statement of the proposition. 
	That $\mathcal{L}_d \to \mathcal{H}_d$ is finite \'{e}tale is well-known;
	see for example \cite[Prop.~3.6]{Sch85}.
	Next assume that $\mathcal{L}_d \times_\QQ k \to \mathcal{H}_d \times_\QQ k$ has
	reducible generic fibre, for some number field $k$. As $\mathcal{L}_d \to \mathcal{H}_d$ is finite \'{e}tale,
	we deduce that every degree $d$ del Pezzo surface
	over $k$ has reducible Hilbert scheme of lines, which contradicts the above.
	Hence $L$ is well-defined. We also deduce that $L \cap \bar{\QQ} = \QQ$ and that $L\otimes_\QQ k$ is a field for all 
	number fields $k$.
	
	Next note that $\Gal(K/H) \subset W(\E_{9-d})$. 
	However, as $\mathcal{L}_d \to \mathcal{H}_d$ is finite \'{e}tale
	and there exists a degree $d$ del Pezzo surface over $\QQ$ whose splitting field
	has Galois group $W(\E_{9-d})$, we find that actually $\Gal(K/H) \cong W(\E_{9-d})$.
	To see that $K \cap \bar{\QQ} = \QQ$, assume for a contradiction that 
	$K \cap \bar{\QQ} = k$, for some non-trivial number field $\QQ \subset k$. Then, as 
	$k \subset K$ and $[K:H] = \#W(\E_{9-d})$, the Galois closure of $(L\otimes_\QQ k) / (H \otimes_\QQ k)$
	has degree strictly smaller than $\#W(\E_{9-d})$. This contradicts the fact that
	there is a del Pezzo surface of degree $d$ over $k$ whose splitting field has degree $\#W(\E_{9-d})$.
\end{proof}

We now apply Ekedahl's version of the Chebotarev density theorem  \cite[Lem.~1.2]{Eke98}.

\begin{proposition} \label{prop:Hilbert}
	Let $d\leq 6$ and let $C$ be a conjugacy class of $W(\E_{9-d})$. Then
	\begin{equation} \label{eqn:tau_limit}
		\lim_{q \to \infty}\frac{\#\{ S \in \mathcal{H}_d(\FF_q) : C(S) = C\}}
		{\#\mathcal{H}_d(\FF_q)} = \frac{\# C}{\# W(\E_{9-d})}.
	\end{equation}
\end{proposition}
\begin{proof}
	Let $\mathcal{K}_d$ be the Galois closure of $\mathcal{L}_d \to \mathcal{H}_d$.
	Let $q$ be a prime power and choose a number field $k$ which admits a prime ideal $\fp$ of norm $q$.
	Proposition \ref{prop:etale} implies that the map
	$\mathcal{K}_d \times_\ZZ \OO_k \to \mathcal{H}_d \times_\ZZ \OO_k$ has Galois group $W(\E_{9-d})$
	and that the map
	$\mathcal{K}_d \times_\ZZ \OO_k \to \Spec \OO_k$ has geometrically irreducible generic
	fibre. We may therefore apply \cite[Lem.1.2]{Eke98} to deduce that 
	$$\frac{\#\{ S \in \mathcal{H}_d(\FF_q) : C(S) = C\}}
		{\#\mathcal{H}_d(\FF_q)} = \frac{\# C}{\# W(\E_{9-d})} + O(q^{-1/2}),$$
	where the implied constant is independent of $q$ and $k$. 
	The result follows.
\end{proof}

\subsection{Proof of Theorem \ref{thm:inverse}}
The group scheme $\mathcal{G}_d$ acts on $\mathcal{H}_d$ in a natural way, with two elements lying in the same orbit if and only if they are isomorphic. Using this and
Proposition \ref{prop:Hilbert}, we obtain
\begin{align*}
	\lim_{q \to \infty}\frac{\sum_{S \in \mathcal{S}_d(\FF_q), C(S) = C}
	\frac{1}{|\Aut S|}}{\sum_{S \in \mathcal{S}_d(\FF_q)}\frac{1}{|\Aut S|}}
	& = \lim_{q \to \infty}\frac{\sum_{S \in \mathcal{H}_d(\FF_q)/\mathcal{G}_d(\FF_q) , C(S) = C}
	\frac{1}{|\Aut S|}}{\sum_{S \in \mathcal{H}_d(\FF_q)/\mathcal{G}_d(\FF_q) }\frac{1}{|\Aut S|}}\\
	& = \lim_{q \to \infty}\frac{\#\{ S \in \mathcal{H}_d(\FF_q)
	: C(S) = C\}}{\#\mathcal{H}_d(\FF_q)} \\	
	& =  \frac{\# C}{\# W(\E_{9-d})}. 
\end{align*}

\subsection{Proof of Theorem \ref{thm:inverse_2}}
We first use known results on automorphism groups of generic blow ups of $\PP^2$ \cite{Koi88} to show that the automorphism group is constant on some open subset of $\mathcal{H}_d$ for $d \leq 3$.

\begin{lemma} \label{lem:aut}
	Let $d \leq 3$, let $a_3 = 1$ and $a_1=a_2=2$. There exists an open subscheme
	$\mathcal{U}_d \subset \mathcal{H}_d$, such that $\mathcal{U}_d \to \Spec \ZZ$
	is surjective and
	$$S \in \mathcal{U}_d \implies |\Aut S| = a_d.$$
\end{lemma}
\begin{proof}
	Let $\mathcal{A}_d\to \mathcal{H}_d$ be the relative automorphism group scheme for the
	universal family of degree $d$ del Pezzo surfaces over $\mathcal{H}_d$.
	This embeds into $\mathcal{G}_d$ as a closed
	subgroup scheme, hence has finite type over $\ZZ$. 
	The main theorem of \cite{Koi88} implies that for each $x \in \Spec \ZZ$,
	the generic fibre of $\mathcal{A}_{d,\kappa(x)} \to 
	\mathcal{H}_{d,\kappa(x)}$
	is a finite scheme of degree $a_d$, where $\kappa(x)$ denotes the residue field of $x$.
	Moreover, when $d=2$ or $1$  it even 
	has $2$ rational points (corresponding to the Geiser and Bertini involution,
	respectively). As $\mathcal{A}_d \to \mathcal{H}_d$ is of finite type,
	the result follows by spreading out.
\end{proof}

We now prove Theorem \ref{thm:inverse_2}.
Let $\mathcal{U}_d \subset \mathcal{H}_d$ and $a_d$ be as in Lemma \ref{lem:aut}. As $\mathcal{H}_{d,\FF_p}$ is geometrically
integral for all primes $p$, the Lang-Weil estimates \cite{LW54} imply that, in the limit \eqref{eqn:tau_limit}, 
proper Zariski closed subsets are negligible. Hence applying Proposition \ref{prop:Hilbert} and Lemma \ref{lem:aut},
we obtain
\begin{align*}
	\lim_{q \to \infty} \frac{\#\left\{S \in \mathcal{S}_d(\FF_q) : C(S) = C\right\}}{\# \mathcal{S}_d(\FF_q)}
	& = \lim_{q \to \infty}\frac{\#\{ S \in \mathcal{H}_d(\FF_q)/\mathcal{G}_d(\FF_q) 
	: C(S) = C\}}{\#(\mathcal{H}_d(\FF_q)/\mathcal{G}_d(\FF_q))} \\	
	& = \lim_{q \to \infty}\frac{\#\{ S \in \mathcal{U}_d(\FF_q)/\mathcal{G}_d(\FF_q) : 
	 C(S) = C\}}{\#(\mathcal{U}_d(\FF_q)/\mathcal{G}_d(\FF_q))} \\	
	& = \lim_{q \to \infty}\frac{\#\{ S \in \mathcal{U}_d(\FF_q) : 
	 C(S) = C\}/a_d}{\#\mathcal{U}_d(\FF_q)/a_d} \\	
	& = \lim_{q \to \infty}\frac{\#\{ S \in \mathcal{H}_d(\FF_q) : C(S) = C\}}{\#\mathcal{H}_d(\FF_q)} \\
	& =  \frac{\# C}{\# W(\E_{9-d})}. 
\end{align*}

\section{A corrected version of Manin's and Swinnerton-Dyer's table} \label{sec:table}
In this section we present our corrected version of Manin's \cite[Tab.~1,p.~176]{Man86} and Swinnerton-Dyer's
\cite[Tab.~1]{SD67} tables on cubic surfaces over finite fields. 

\subsection{Notation}
We mostly use the same notation as Manin, except for Column $8$ (see \S \ref{sec:C8}),
and a new Column $10$, which relates to Urabe's tables on del Pezzo surfaces of degree $1$ and $2$ \cite{Ura96a}. We highlight by $*$ a mistake in Manin's table and give the correct value (see Section \ref{sec:mistakes} for detailed explanations on these corrections).

Note that, by Corollary \ref{cor:inverse}, every conjugacy class of $W(\Esix)$ arises from some smooth cubic surface over $\FF_q$ for all
sufficiently large $q$. We may therefore describe the columns in terms of the geometry of such a surface $S$. 
We take $q$ sufficiently large and denote by $\Frob_q$ the Frobenius element of $\Gal(\bar{\FF}_q/\FF_q)$.

\begin{enumerate}
	\item[0.] This column denotes the numbering of the relevant conjugacy class, in the notation 
	of Frame \cite{Fra51} and 	Swinnerton-Dyer \cite{SD67}.
	\item[1.] This is Manin's number for the conjugacy class.
	\item[2.] The index $i(S)$ is defined to be the size of the largest Galois invariant 
	collection of pairwise skew lines over $\bar{\FF}_q$.
	\item[3.] The order of an element in the conjugacy class.
	\item[4.] The reciprocal measure $\mu(C)^{-1}$ of $C$, defined to be $\#W(\Esix)/\#C$.
	\item[5.] The  eigenvalues of $\Frob_q$ acting on $(\Pic \bar{S}) \otimes \CC$.
	The notation $n^b$ means that there are exactly $b$ eigenvalues that are primitive $n$th roots of unity.
	\item[6.] The trace $a(S)$ of $\Frob_q$ acting on $\Pic \bar{S}$.
	\item[7.] The Galois cohomology group $\mathrm{H}^1(\FF_q, \Pic \bar{S})$. We use the notation $n^2$
	to denote the group $(\ZZ/n\ZZ)^2$.
	\item[8.] The orbit type of $\Gal(\bar{\FF}_q/\FF_q)$ on the lines of $\bar{S}$ (see \S \ref{sec:C8}).
	\item[9.] Information about the blow-downs of $S$. If $S$ is minimal  this column is empty. For $i(S)=1$,
	we give Manin's decomposition type \cite[Tab.~2]{Man86} of the quartic del Pezzo surface obtained
	by blowing down a line.
	For $i(S)>1$, we give an orbit type of a collection of $i(S)$
	Galois invariant skew lines over $\bar{\FF}_q$.
	\item[10.] The conjugacy class of the del Pezzo surface of degree $2$ in Urabe's table 
	\cite[Tab.~1]{Ura96a}
	obtained by blowing-up $S$ in a rational point not on a line.
\end{enumerate}
We have confirmed that the decompositions claimed by Manin
for quartic del Pezzo surfaces in \cite[Tab.~2]{Man86} and \cite[Tab.~3]{Man86} are correct in the latest
edition of his book.

\subsubsection{Column $8$} \label{sec:C8}
In our table we use new notation to describe the orbits. 
Both Swinnerton-Dyer \cite{SD67} and Manin \cite{Man86} only wrote down the size of each orbit, and separated
orbits of the same size and different ``configuration''. For example, for No.~5
they wrote $3,6^3, 6$ to denote an orbit of length $3$ and four orbits of length $6$,
of which three have the same configuration and the other a different configuration.

We take a slightly different approach, which makes clear which \emph{exact} configurations these correspond to.
First recall that to each cubic surface $S$ we can associate a graph as follows: one vertex for each line,
with two vertices being adjacent if and only if the corresponding lines meet. This graph is independent
of $S$, up to isomorphism, and is the so-called \emph{Schl\"{a}fli graph} $G$, which has
automorphism group $W(\E_6)$ \cite[Thm.~23.9]{Man86}. 
Let now $C \subset W(\Esix)$ be a conjugacy class and $w \in C$, considered with its action on $G$. Let $O$ be an orbit of vertices for $w$. We define the \emph{type} of $O$
to be the isomorphism class of the induced subgraph of $G$ determined by $O$. We define the \emph{orbit type} of $C$ to be the 
multiset of types of its orbits for some $w \in C$; this is independent of the choice of $w$, up to isomorphism. The graphs which arise this way are quite special;
by definition they all admit a vertex-transitive action of a cyclic group. Graph theorists call such graphs
\emph{Circulant graphs}.

We notate these graphs as follows. The notation $n_{m_1,\ldots,m_s}^b$ denotes $b$ copies of
the circulant graph with vertices labelled $0, \ldots, n-1$, and vertex $0$ being adjacent 
to the vertices $\pm m_1,\ldots, \pm m_s \bmod n$
(this determines all the edges of the graph). E.g.
\begin{align*}
		n &= \text{Edgeless graph of order $n$}, \quad n_1 = \text{Cycle graph of order $n$}.
\end{align*}
The following well-known graphs occur in our table.
\begin{align*}
	6_{1,3} &= \text{Utility graph}, \quad 6_{2,3} = \text{Triangular prism}, \\
	8_{1,4} &= \text{Wagner graph}, \quad 10_{1,3} = \text{5-crown graph}.
\end{align*}
Urabe \cite{Ura96a} separates orbits by their \emph{characteristic sequences}, which is essentially
the list of vertices which meet the vertex $0$. However, different characteristic
sequences may give rise to isomorphic graphs, e.g.~$9_{1,3} \cong 9_{2,3} \cong 9_{3,4}$ (in the above notation).
In particular, the reader should bear in mind that some of the orbits in Urabe's table which are separated by their
characteristic sequences have isomorphic graphs.

\begin{table}[ht]
\renewcommand\thetable{7.1}
\makebox[\textwidth][c]{

$\begin{array}{lllllllllll}
 \hline
 0. & 1. & 2.  & 3.  &  4.  & 5.  & 6. 
 & 7. & 8. &9. \text{ Blow} &10. \\
 \text{Class} & \text{No.} & i(S) & \text{Order} & \mu(C)^{-1} 
 & \text{Eigenvalues} & a(S) &  \mathrm{H}^1 & \text{Orbit type} & \text{down} & \text{Blow-up} \\
 \hline
 C_{13} & 1  & 0   & 12 & 12    & 1,3^2,12^4    & 0  & 0   & 3_1, 12^2_{1,4,6} 		&   	& 22\\
 C_{12} & 2  & 0   & 6  & 72    & 1,3^2,6^4     & 2  & 0   & 3_1, 6_{2,3}^4  		& 		& 24\\
 C_{11} & 3  & 0   & 3  & 648   & 1,3^6         & -2 & 3^2 & 3_1^9  				&		& 20\\
 C_{14} & 4  & 0   & 9  & 9     & 1,9^6         & 1  & 0   & 9_{1,3}^3 				& 		& 23\\
 C_{10} & 5  & 0   & 6  & 36    & 1,2^2,3^2,6^2 & -1 & 2^2 & 3_1,6_{1,3}^3,6_{2,3} 	&		& 21\\
 C_{24} & 6  & 1   & 12 & 12    & 1^2,2,4^2,6^2 & 2  & 0   & 1,4_2,4_1,6_3,12_{2,3} & \text{I}	& 33\\
 C_{20} & 7  & 1   & 8  & 8     & ^* 1^2,2,8^4  & 1  & ^*0 & 1,2_1,8_4,8_{1,4}^2 	& \text{XVIII}	& 32\\
 C_{7}  & 8  & 1   & 6  & 36    & 1^3,2^2,6^2   & 2  & 0   & 1^3,2_1^3,6_3^3 		& \text{II}	& 26\\
 C_{19} & 9  & 1   & 4  & 96    & 1^2,2^3,4^2   & -1 & 2^2 & 1,2_1^3,4_2,4_1^4 		& \text{X} 	& 29\\
 C_{4}  & 10 & 1   & 4  & 96    & 1^3,4^4       & 3  & ^*0 & 1^3,4_2^6 				& \text{V} 	& 27\\
 C_{3}  & 11 & 1   & 2  & 1152  & 1^3,2^4       & -1 & 2^2 & 1^3,2_1^{12} 			& \text{IV}	& 25\\
 C_{25} & 12 & ^*5 & 10 & 10    & 1^2,2,5^4     & 0  & 0   & 2,5,5_1^2,10_{1,3} 	& 5		& 58\\
 C_{22} & 13 & 3   & 6  & 36    & 1^2,2,3^4     & -1 & 0   & 3^2,3_1^3,6_1^2		& 3		& 42\\
 C_{8}  & 14 & ^*5 & 6  & 24    & 1^3,2^2,3^2   & 0  & 0   & 1,2^2,2_1^2,3^2,6_{1}^2& 2,3	& 53\\
 C_{23} & 15 & 6   & 6  & 12    & 1^2,2,3^2,6^2 & 1  & 0   & 3_1,6^2,6_{1,3},6_{2,3}& 6 	& 59\\
 C_{15} & 16 & 6   & 5  & 10    & 1^3,5^4       & 2  & 0   & ^*1^2,5^3,5_1^2 		& 1,5	& 56\\
 C_{5}  & 17 & 6   & 4  & 16    & 1^3,2^2,4^2   & 1  & 0   & 1,2^2,2_1,4^2,4_2,4_1^2& 2,4 	& 55\\
 C_{9}  & 18 & 6   & 3  & 108   & 1^3,3^4       & 1  & 0   & 3^6,3_1^3 				& 3,3	& 54\\
 C_{18} & 19 & 6   & 4  & 32    & 1^4,2,4^2     & 3  & 0   & 1^5,2_1,4^4,4_2 		& 1^2,4	& 52\\
 C_{21} & 20 & 6   & 6  & 36    & 1^4,2,3^2     & 2  & 0   & 1^3,2^3,3^4,6_{1} 		& 1,2,3	& 51\\
 C_{17} & 21 & 6   & 2  & 96    & 1^4,2^3       & 1  & 0   & 1^3,2^6,2_1^6 			& 2^3	& 50\\
 C_{6}  & 22 & 6   & 3  & 216   & 1^5,3^2       & 4  & 0   & 1^9,3^6 				& 1^3,3	& 49\\
 C_{2}  & 23 & 6   & 2  & 192   & 1^5,2^2       & 3  & 0   & 1^7,2^8,2_1^2 			& 1^2,2^2&48\\
 C_{16} & 24 & 6   & 2  & 1440  & 1^6,2         & 5  & 0   & 1^{15},2^6 			& 1^4,2	& 47\\
 C_{1}  & 25 & 6   & 1  & 51840 & 1^7           & 7  & 0   & 1^{27} 				& 1^6	& 46\\
\end{array}$}

 \caption{Conjugacy classes of $W(\Esix)$.}
 \label{tab:Manin}
\end{table} 

\subsection{Compilation}
We compiled the table using \texttt{Magma} \cite{Magma}. The code is available on the authors' web pages.
We constructed the Schl\"{a}fli graph, together with its action of $W(\E_6)$,
using the method outlined in \cite[\S3.1.1]{JL15}. One calculates the conjugacy classes, order, measure, index and 
orbit type using standard commands for permutation groups. The eigenvalues and trace were found by constructing the relevant representation of $W(\E_6)$. 
We calculated $\mathrm{H}^1(\FF_q, \Pic \bar{S})$ using \cite[Prop.~31.3]{Man86}.

We identified which degree $2$ del Pezzo surface in \cite[Tab.~1]{Ura96a} is obtained by blowing-up a rational
point not on a line by comparing the eigenvalues
(these can be deduced from the ``Frame symbol'' of \emph{loc.~cit}). Blowing-up
a rational point multiplies the characteristic polynomial by $(x-1)$.
A conjugacy class of $W(\Esix)$
is uniquely determined by its eigenvalues,
however this is no longer true
for  $W(\Eseven)$. Nonetheless, for many characteristic
polynomials  there is a unique conjugacy class, and one finds more than 
one potential conjugacy class of $W(\Eseven)$ only when blowing-up No.'s $5,9,11,15,17,21$ from Table \ref{tab:Manin}.
For these classes there are exactly two conjugacy
classes of $W(\Eseven)$ with the same characteristic polynomial.
However, in all these cases only one  contains a line,
which isolates the relevant class.

\subsection{Corrections} \label{sec:mistakes}
We now explain the issues with Manin's table.

\subsubsection*{Typesetting}
There are some typesetting issues with  Columns $5$ and $8$
which make them difficult to read; we have fixed these  in our table.

\subsubsection*{Column $5$:}
The mistake here is a minor typo due to recording the incorrect value from Swinnerton-Dyer's table \cite{SD67},
and easily corrected.

\subsubsection*{Column $7$:}
The errors regarding the calculation of $\mathrm{H}^1(\FF_q, \Pic \bar{S})$
were first discovered by Urabe \cite{Ura96a}; we have nothing new to add.

\smallskip

We now explain the mathematical mistakes we found, regarding the index and the orbit type.
The corrected values appear here for the first time.

\subsubsection*{Column $2$:}
Let $q$ be a sufficiently large prime power.

\begin{itemize}
\item[No.~12:] Let $S'$ be a smooth non-split quadric surface
over $\FF_q$. Let $S$ be the cubic surface obtained by blowing-up $S'$ in a closed point of degree
$5$ in general position. One easily sees that the splitting field of $S$ has degree $10$,
hence such a surface must have class 12. But, by construction, one may contract 
an orbit of length $5$ on $S$, hence the index is not $2$, as claimed by Manin.

\item[No.~14:] Again let $S'$ be a smooth non-split quadric surface
over $\FF_q$. Let $S$ be the blow-up of $S'$ in closed points of degree $2$ and $3$
in general position. This surface has splitting field of degree $6$ and $a(S) = 0$,
thus it must have class 14. As before, we see that the index is not $3$, as claimed
by Manin.
\end{itemize}

These constructions show that such surfaces $S$ have index at least $5$. However, in both 
cases we have $a(S) = 0$, hence Lemma \ref{lem:blow_up} implies that the index is not $6$, so
the index is indeed $5$ as claimed.

\subsubsection*{Column $8$:}
We found an error regarding the orbit type of No.~16. This mistake can be traced
back to \cite[Tab.~1]{SD67}. Swinnerton-Dyer claims that the configuration is $1^2,5^2,5^2,5$
(in his notation), however it is in fact $1^2,5^3,5_1^2$ (in our notation).

Consider the blow-up of $\PP^2$ in a rational
point and a closed point of degree $5$ in general position over $\FF_q$.
The splitting field has degree $5$, hence by Table \ref{tab:Manin}
it corresponds
to the class $16$. The orbits of length $1$ come from the rational point
and the conic passing through the quintic points. The orbits of length $5$ arise from:
\begin{enumerate}
	\item $5$ lines above the quintic points.
	\item $5$ lines passing through the rational point and one of the quintic points.
	\item $5$ conics passing through the rational point and four of the quintic points.
	\item $5$ lines passing through adjacent quintic points.
	\item $5$ lines passing through non-adjacent quintic points.
\end{enumerate}

One checks that the first three types consist of pairwise skew lines,
whereas the last two types consist of lines meeting exactly two others
(this can be deduced from \cite[Rem.~V.4.10.1]{Har77}, for example). This gives rise
to orbit type $1^2,5^3,5_1^2$, as claimed.

\smallskip

Our guess is that Urabe over-looked the remaining mistakes in Manin's table as they concern
the index and orbit type, which behave erratically with respect to blow-ups.

\end{document}